\documentclass[12pt,a4paper]{article}
\usepackage{amssymb}
\usepackage{amsmath}
\usepackage[symbol]{footmisc}
\usepackage{amsthm}
\usepackage{amscd}
\usepackage{amstext}
\usepackage{vmargin}
\usepackage{fancyhdr}
\setpapersize{A4}
\setmarginsrb{3.0cm}{1.5cm}{3.0cm}{3cm}{1cm}{.5cm}{2cm}{2cm}

\usepackage{amsfonts}
\usepackage{amsthm}
\usepackage[numbers]{natbib}
\usepackage[latin1]{inputenc}
\usepackage[T1]{fontenc}
\usepackage{enumerate}
\usepackage[american]{babel}
\usepackage{graphicx}
\usepackage{bbm}
\usepackage[mathscr]{euscript}
\usepackage{mathrsfs}
\usepackage{stmaryrd}
\usepackage{soul}
\usepackage{hyperref}
\usepackage{xspace}
\usepackage[draft,danish]{fixme}
\usepackage{mathtools}
\usepackage{dsfont}
\usepackage{url}
\usepackage[ps,all,dvips]{xy}
\SelectTips{cm}{12}
\CompileMatrices

  \makeatletter
 \useshorthands{"}
 \defineshorthand{"-}{\nobreak-\bbl@allowhyphens}
 \makeatother

\theoremstyle{plain}
\newtheorem{lemma}{Lemma}[section]

\newtheorem{theorem}[lemma]{Theorem}
\newtheorem{corollary}[lemma]{Corollary}
\theoremstyle{definition}
\newtheorem{example}[lemma]{Example}

\theoremstyle{definition}

\theoremstyle{remark}
\newtheorem{remark}[lemma]{Remark}

\newcommand{\devnull}[1]{}

\numberwithin{equation}{section}

\title{Equivalent martingale measures for L\'{e}vy-driven moving averages and related processes}
\author{Andreas Basse-O'Connor, Mikkel Slot Nielsen and Jan Pedersen}
\date{\small Department of Mathematics, Aarhus University,\\
 \{basse, mikkel, jan\}@math.au.dk}
	
\begin{document}
\maketitle

\begin{abstract}
In the present paper we obtain sufficient conditions for the existence of equivalent martingale measures for L\'{e}vy-driven moving averages and other non-Markovian jump processes. The conditions that we obtain are, under mild assumptions, also necessary. For instance, this is the case for moving averages driven by an $\alpha$-stable L\'{e}vy process with $\alpha \in (1,2]$. 

Our proofs rely on various techniques for showing the martingale property of stochastic exponentials.
\end{abstract}

\section{Introduction and a main result}\label{introduction}
Absolutely continuous change of measure for stochastic processes is a classical problem in probability theory and there is a vast literature devoted to it. One motivation is the fundamental theorem of asset pricing, see Delbaen and Schachermayer \cite{fundametal_thm_asset}, which relates existence of an equivalent martingale measure to absence of arbitrage (or, more precisely, to the concept of no free lunch with vanishing risk) of a financial market. Several sharp and general conditions for absolutely continuous change of measure are given in \cite{dCriens,kallsenshi,LepingleMemin,ProtterShimbo}, and in case of Markov processes and solutions to stochastic differential equations, strong and explicit conditions are available, see e.g. \cite{DitoYor,Dawson,EberleinJacod,KabanovLiptser,MijaUru} and references therein. 

The main aim of the present paper is to obtain explicit conditions for the existence of an equivalent martingale measure (EMM) for L\'{e}vy-driven moving averages, and these are only Markovian in very special cases. Moving averages are important in various fields, e.g.\ because they are natural to use when modelling long-range dependence (for other applications, see \cite{podolskij2015ambit}). Recalling that Hitsuda's representation theorem characterizes when a Gaussian process admits an EMM, see \cite[Theorem~6.3']{Hida}, and L\'{e}vy-driven moving averages are infinitely divisible processes, our study can also be seen as a contribution to a similar representation theorem for this class.

We will now introduce our framework. Consider a probability space $(\Omega, \mathscr{F}, \mathbb{P})$, on which a two-sided L\'{e}vy process $L=(L_t)_{t\in \mathbb{R}}$, with $L_0=0$, is defined. Let $T>0$ be a fixed time horizon and let $(\mathscr{F}_t)_{t\in [0,T]}$ be the smallest filtration that satisfies the usual conditions, see \cite[Definition~1.3~(Ch.~I)]{JS}, and such that
\begin{align*}
\sigma \left(L_s\ :\ -\infty <s\leq t \right)\subseteq \mathscr{F}_t
\end{align*}
for $t \in [0,T]$. Define a stationary $L$-driven moving average $(X_t)_{t\in [0,T]}$ by
\begin{align}\label{movingAverage}
X_t = \int_{-\infty}^t \varphi (t-s) dL_s, \quad t \in [0,T],
\end{align}
for a given function $\varphi:\mathbb{R}_+\to \mathbb{R}$, such that the integral in (\ref{movingAverage}) is well-defined. To avoid trivial cases, suppose further that the set of $t\geq 0$ with $\varphi (t) \neq 0$ is not a Lebesgue null set. Our aim is to find explicit conditions that ensure the existence of a probability measure $\mathbb{Q}$ on $(\Omega,\mathscr{F})$, equivalent to $\mathbb{P}$, under which $(X_t)_{t\in [0,T]}$ is a local martingale. Furthermore, we are interested in the structure of $(X_t)_{t \in [0,T]}$ under $\mathbb{Q}$. 

A necessary condition for a process to admit an EMM is that it is a semimartingale, and this property is (under mild assumptions on the L\'{e}vy measure) characterized for $L$-driven moving averages in Basse-O'Connor and Rosinski \cite{abocjr} and Knight \cite{Knight}. Other relevant references in this direction include \cite{abocjp,Patrick}. In the case where $L$ is Gaussian, and relying on Knight \cite[Theorem~6.5]{Knight}, Cheridito \cite[Theorem~4.5]{Patrick} gives a complete characterization of the $L$-driven moving averages that admit an EMM:
\begin{theorem}[P. Cheridito]\label{GaussianEMM}
Suppose that $L$ is a Brownian motion. Then the moving average $(X_t)_{t\in [0,T]}$ defined in (\ref{movingAverage}) admits an EMM if and only if $\varphi (0) \neq 0$ and $\varphi$ is absolutely continuous with a density $\varphi'$ satisfying $\varphi' \in L^2 (\mathbb{R}_+)$.
\end{theorem}
Despite that, in general, the existence of an EMM is a stronger condition than being a semimartingale, Theorem~\ref{GaussianEMM} shows (together with \cite[Theorem~6.5]{Knight}) that for Gaussian moving averages of the form (\ref{movingAverage}), the two concepts are equivalent when $\varphi(0)\neq 0$. If $L$ has a non-trivial L\'{e}vy measure, explicit conditions for the existence of an EMM have, to the best of our knowledge, not been provided. It would be natural to try to obtain such conditions using the same techniques as in Theorem~\ref{GaussianEMM}. However, these techniques are based on a local version of the Novikov condition, which will not be fulfilled as soon as the driving L\'{e}vy process is non-Gaussian. This is an implication of the fact that $\int_\mathbb{R} e^{\varepsilon x^2} \varpi (dx) = \infty$ for any $\varepsilon >0$ and any non-Gaussian infinitely divisible distribution $\varpi$, see \cite[Theorem~26.1]{Sato}. Consequently, to prove the existence of an EMM in a non-Gaussian setting, a completely different approach has to be used. An implication of our results is the non-Gaussian counterpart of Theorem~\ref{GaussianEMM} which is formulated in Theorem~\ref{nonGaussianEMM} below. In this formulation, $c \geq 0$ denotes the Gaussian component of $L$ and $F$ is its L\'{e}vy measure.

\begin{theorem}\label{nonGaussianEMM}
Let $(X_t)_{t \in [0,T]}$ be a L\'{e}vy-driven moving average given by (\ref{movingAverage}). Suppose that $L$ has sample paths of locally unbounded variation, either $\int_{\{\vert x \vert >1\}}x^2F(dx) < \infty$ or $x \mapsto F((-x,x)^c)$ is regularly varying at $\infty$ of index $\beta \in [-2,-1)$, and
\begin{flalign}\label{fSupp}
\text{the support of } F \text{ is unbounded on both } (-\infty,0) \text{ and } (0,\infty).
\end{flalign}
Then $(X_t)_{t \in [0,T]}$ admits an EMM if and only if $\varphi (0)\neq 0$ and $\varphi$ is absolutely continuous with a density $\varphi'$ satisfying
\begin{align}\label{densityCondition}
c\int_0^\infty \varphi'(t)^2dt + \int_0^\infty \int_\mathbb{R} \vert x \varphi' (t)\vert \wedge (x \varphi'(t))^2F(dx)dt < \infty.
\end{align}
If instead of (\ref{fSupp}), we have that the support of $F$ is contained in a compact set, but is not contained in $(-\infty,0)$ nor $(0,\infty)$, then the remaining assumptions are sufficient to ensure that $(X_t)_{t\in [0,T]}$ admits an EMM provided that $\varphi'$ is bounded.
\end{theorem}

If $L$ is a symmetric $\alpha$-stable L\'{e}vy process with $\alpha \in (1,2)$, $x\mapsto F((-x,x)^c)$ is regularly varying of index $-\alpha$ and condition (\ref{densityCondition}) is equivalent to $\varphi' \in L^\alpha (\mathbb{R}_+)$ (see \cite[Example~4.9]{abocjr}). Thus, since condition (\ref{fSupp}) of Theorem~\ref{nonGaussianEMM} is trivially satisfied as well, we get directly the following natural extension of Theorem~\ref{GaussianEMM}:
\begin{corollary}\label{extendGaussianEMM} Suppose that $L$ is a symmetric $\alpha$-stable L\'{e}vy process with index $\alpha \in (1,2]$. Then the moving average $(X_t)_{t\in [0,T]}$ defined in (\ref{movingAverage}) admits an EMM if and only if $\varphi (0)\neq 0$ and $\varphi$ is absolutely continuous with a density $\varphi'$ satisfying $\varphi' \in L^\alpha (\mathbb{R}_+)$.
\end{corollary}
A result similar to Corollary~\ref{extendGaussianEMM} can be formulated when $L$ is a symmetric tempered stable L\'{e}vy process that is, when the L\'{e}vy measure takes the form $F(dx)=\eta\vert x \vert^{-\alpha-1}e^{-\lambda \vert x \vert}$ for $\eta,\lambda >0$ and $\alpha \in [1,2)$. Indeed, since $\int_\mathbb{R}x^2F(dx) < \infty$ and (\ref{fSupp}) is satisfied in this setup, there exists an EMM $\mathbb{Q}$ for $(X_t)_{t\in [0,T]}$ if and only if $\varphi (0)\neq 0$ and $\int_0^\infty \vert \varphi'(t) \vert^\alpha \wedge \varphi' (t)^2 dt <\infty$ (as the latter condition is equivalent to (\ref{densityCondition}) cf. \cite[Example~4.9]{abocjr}). 

It may be stressed that the Gaussian case considered in Theorem~\ref{GaussianEMM} and the non-Gaussian case considered in Theorem~\ref{nonGaussianEMM} are of fundamental different structure. Indeed, when $L$ is a Brownian motion, one can apply the martingale representation theorem to show that the EMM is unique, and by invariance of the quadratic variation under equivalent change of measure, $(X_t - X_0)_{t\in [0,T]}$ is a Brownian motion under the EMM (one may need a semimartingale decomposition of $(X_t)_{t \in [0,T]}$, see e.g. (\ref{canonicalMA})). If $L$ is a general L\'{e}vy process, Theorem~\ref{Thm1} and Remark~\ref{charAnalysis} in Section~\ref{mainRes} show that the EMM will not be unique, and $(X_t-X_0)_{t \in [0,T]}$ and $(L_t)_{t\in [0,T]}$ will not be L\'{e}vy processes under any of our constructed EMMs. 

Besides the moving average framework we will also, for a general filtration $(\mathscr{F}_t)_{t\in [0,T]}$, study EMMs for semimartingales of the form
\begin{align}\label{processOfInterest}
X_t = L_t + \int_0^t Y_s ds, \quad t \in [0,T],
\end{align}
for a given $(\mathscr{F}_t)_{t\in [0,T]}$-L\'{e}vy process $(L_t)_{t\in [0,T]}$ and a predictable process $(Y_t)_{t\in [0,T]}$ such that $t\mapsto Y_t$ is integrable on $[0,T]$ almost surely. This study turns out to be useful in order to deduce results for moving averages.

We will shortly present the outline of this paper. Section~\ref{mainRes} presents Theorem~\ref{Thm1}, which concerns precise and tractable conditions on $(L_t)_{t\in [0,T]}$ and $(Y_t)_{t\in [0,T]}$ ensuring the existence of an EMM for $(X_t)_{t\in [0,T]}$ in (\ref{processOfInterest}). An implication of this result is Theorem~\ref{nonGaussianEMM} and in turn Corollary~\ref{extendGaussianEMM}. Theorem~\ref{Thm1} is followed by a predictable criterion ensuring the martingale property of stochastic exponentials, Theorem~\ref{Thm2}, and this is based on a general approach of Lépingle and Mémin \cite{LepingleMemin}. Due to the nature of this criterion, it can be used for other purposes than verifying the existence of EMMs for $(X_t)_{t\in [0,T]}$ and thus, the result is of independent interest. Both Theorem~\ref{Thm1} and Theorem~\ref{Thm2} are accompanied by remarks and examples that illustrate their applications. Subsequently, Section~\ref{prel} recalls the most fundamental and important concepts in relation to change of measure and integrals with respect to random measures. These concepts will be used throughout Section~\ref{proofs} which is devoted to prove the statements of Section~\ref{mainRes}. During Section~\ref{proofs} one will also find additional remarks and examples of a more technical nature.
\section{Further main results}\label{mainRes} 
Let $L=(L_t)_{t \in [0,T]}$ be an integrable $(\mathscr{F}_t)_{t \in [0,T]}$-L\'{e}vy process with triplet $(c, F,b^h)$, relative to some truncation function $h:\mathbb{R}\to \mathbb{R}$. Here $b^h \in \mathbb{R}$ is the drift component, $c\geq 0$ is the Gaussian component, and $F$ is the L\'{e}vy measure. Set $\xi = \int_\mathbb{R} (x - h(x))F(dx) + b^h$ so that $\mathbb{E}[L_t] = \xi t$ for $t \in [0,T]$. We denote by $\mu$ the jump measure of $L$ and by $\nu (dt,dx) = F(dx)dt$ its compensator. It will be assumed that $L$ has both positive and negative jumps such that we can choose $a,b>0$ with
\begin{align}\label{truncation}
\min\{F((-b,-a)),F((a,b))\} >0.
\end{align}
In Theorem~\ref{Thm1} we will give conditions for the existence of an EMM $\mathbb{Q}$ for $(X_t)_{t\in [0,T]}$ given by 
\begin{align}\label{processOfInterest2}
X_t = L_t + \int_0^t Y_s ds, \quad t \in [0,T],
\end{align}
where $(Y_t)_{t\in [0,T]}$ is a predictable process and $t \mapsto Y_t$ is Lebesgue integrable on $[0,T]$ almost surely. We will also provide the semimartingale (differential) characteristics of $(X_t)_{t\in [0,T]}$ under $\mathbb{Q}$ (these are defined in \cite[Ch.~II]{JS} and can be found in Section~\ref{prel} as well).

\begin{theorem}\label{Thm1} Let $(X_t)_{t \in [0,T]}$ be given by (\ref{processOfInterest2}). Consider the hypotheses: 
\begin{enumerate}[(h1)]
\item The collection $(Y_t)_{t \in [0,T]}$ is tight and $Y_t$ is infinitely divisible with a L\'{e}vy measure supported in $[-K,K]$ for all $t\in [0,T]$ and some $K>0$. 
\item The L\'{e}vy measure of $L$ has unbounded support on both $(-\infty,0)$ and $(0,\infty)$. 
\end{enumerate}
If either (h1) or (h2) holds, there exists an EMM $\mathbb{Q}$ on $(\Omega, \mathscr{F})$ for $(X_t)_{t \in [0,T]}$ such that $d\mathbb{Q}=\mathscr{E}((\alpha -1)\ast (\mu - \nu))_Td \mathbb{P}$ for some predictable function $\alpha:\Omega \times [0,T]\times \mathbb{R}\to (0,\infty)$, and the differential characteristics of $(X_t)_{t\in [0,T]}$ relative to $h$ under $\mathbb{Q}$ are of the form
\begin{align}\label{characteristics}
\biggr(c, \alpha(t,x)F(dx),b^h + Y_t + \int_\mathbb{R}(\alpha(t,x) - 1)h(x)F(dx)\biggr)
\end{align}
for $t \in [0,T]$.

For any $a,b>0$ that meet (\ref{truncation}), depending on the hypothesis, $\mathbb{Q}$ can be chosen such that:
\begin{enumerate}[(h1)]
\item The function $\alpha$ is explicitly given by
\begin{align}\label{alphaH1}
\alpha(t,x) &= 1+ \frac{(Y_t +\xi)^- x}{\sigma^2_+} \mathds{1}_{(a,b)}(x) - \frac{(Y_t + \xi)^+ x}{\sigma_-^2}\mathds{1}_{(a,b)} (-x)
\end{align}
where $\sigma^2_{\pm} =\int_\mathbb{R} y^2 \mathds{1}_{(a,b)}(\pm y) F(dy)$.
\item With $\lambda = F([-a,a]^c)$, the relations
\begin{align}\label{alphaH2}
\int_{[-a,a]^c} \alpha(t,x) F(dx)=\lambda \quad \text{and} \quad \int_{[-a,a]^c} x\alpha (t,x) F(dx) = - (Y_t + b^h)
\end{align}
hold pointwise, and $\alpha (t,x)=1$ whenever $\vert x\vert  \leq a$.
\end{enumerate}
\end{theorem}
\begin{remark}\label{charAnalysis}
Suppose that Theorem~\ref{Thm1} is applicable. Observe that, for instance by varying $a,b>0$, an EMM for $(X_t)_{t \in [0,T]}$ is not unique. In the following, fix an EMM $\mathbb{Q}$ for $(X_t)_{t \in [0,T]}$, under which its characteristics have a differential form as in (\ref{characteristics}) relative to a truncation function $h$. As a first comment we see that, as long as $(Y_t)_{t \in [0,T]}$ is not deterministic, the characteristic triplet under $\mathbb{Q}$ of both $(L_t)_{t \in [0,T]}$ and $(X_t)_{t\in [0,T]}$ will not be deterministic. Consequently by \cite[Theorem~4.15~(Ch.~II)]{JS}, none of them have independent increments, in particular they will never be L\'{e}vy processes, under $\mathbb{Q}$. 

Despite the fact that $(X_t)_{t\in [0,T]}$ does not have independent increments under $\mathbb{Q}$ we may still extract some useful information from the differential characteristics. Indeed, according to \cite[Theorem~2.34~(Ch.~II)]{JS}, we may represent $(X_t)_{t \in [0,T]}$ through its canonical decomposition (under $\mathbb{P}$) as
\begin{align}\label{canonicalRep}
X_t = X_t^c +  h(x)\ast (\mu - \nu)_t +  (x-h(x))\ast \mu_t + \int_0^t (Y_s + b^h)ds, \quad t \in [0,T], 
\end{align}
where $(X_t^c)_{t \in [0,T]}$ is the continuous martingale part of $(X_t)_{t \in [0,T]}$ under $\mathbb{P}$ and $\ast$ denotes integration, see Section~\ref{prel} for more on the notation. Furthermore, recall that $b^h \in \mathbb{R}$ is the drift component of $(L_t)_{t \in [0,T]}$ relative to $h$ and $\mu$ is the jump measure associated to $(X_t)_{t\in [0,T]}$ (or equivalently, $(L_t)_{t\in [0,T]}$). Consider the specific truncation function $h(x) = x \mathds{1}_{(a,b)^c}(\vert x\vert)$ under (h1) and $h(x) = x \mathds{1}_{[-a,a]}(x)$ under (h2). From (\ref{characteristics}) and (\ref{canonicalRep}) we deduce under $\mathbb{Q}$:
\begin{enumerate}[(i)]
\item The process $X_t^c$, $t \in [0,T]$, remains an $(\mathscr{F}_t)_{t\in [0,T]}$-Brownian motion with variance $c$.
\item It still holds that $h(x) \ast (\mu - \nu)_t$, $t\in [0,T]$, is a zero-mean $(\mathscr{F}_t)_{t \in [0,T]}$-L\'{e}vy process and its distribution is unchanged.
\item The process
\begin{align}\label{measureAffect}
(x-h(x))\ast \mu_t + \int_0^t (Y_s+ b^h)ds, \quad t \in [0,T],
\end{align}
is a local martingale, since $(X_t)_{t\in [0,T]}$ is a local martingale.
\item Except for the drift term involving $(Y_t)_{t\in [0,T]}$, it follows that the only component in (\ref{canonicalRep}) affected by the change of measure (under any of the hypotheses) is $(x-h(x))\ast \mu_t$, $t \in [0,T]$, which goes from a compound Poisson process under $\mathbb{P}$ to a general c\'{a}dl\'{a}g and piecewise constant process under $\mathbb{Q}$. Specifically, it will be affected in such a way that it is compensated according to (\ref{measureAffect}). By exploiting the structure of the compensator of $\mu$ under $\mathbb{Q}$ it follows that the jumps of $(x-h(x))\ast \mu_t$, $t\in [0,T]$, still arrive according to a Poisson process (with the same intensity as under $\mathbb{P}$) under (h2) while under (h1), they will arrive according to a counting process with a stochastic intensity. The (conditional) jump distribution is obtained from Lemma~\ref{nuSpec}.
\end{enumerate}
Note that although, strictly speaking, the function $h(x) = x \mathds{1}_{(a,b)}(\vert x \vert )$ is not a genuine truncation function, we are allowed to use it as such, since $\int_{\{\vert x \vert >1\}} \vert x \vert F(dx) <\infty$ by assumption, which means the integrals in (\ref{canonicalRep}) will be well-defined.
\end{remark}

\begin{remark}\label{thm1assump}
As a first comment on the hypotheses presented in the statement of Theorem~\ref{Thm1} we see that none of them is superior to the other one. Rather, there is a trade off between the restrictions on $(L_t)_{t\in [0,T]}$ and on $(Y_t)_{t\in [0,T]}$. In line with Remark~\ref{LMassumptions}, one may as well replace (h1) by
\begin{enumerate}[(h1')]
\item For any $t\in [0,T]$ and a suitable $\varepsilon >0$, $Y_t \overset{\mathscr{D}}{=}Y_0$ and $\mathbb{E}\big[e^{\varepsilon \vert Y_0 \vert \log (1+ \vert Y_0 \vert)}\big]<\infty$.
\end{enumerate}
The advantage of this hypothesis is that one is not restricted to the case where $Y_t$ is infinitely divisible, however the price to pay is to require that $Y_t \overset{\mathscr{D}}{=} Y_0$ rather than the much weaker assumption of $(Y_t)_{t\in [0,T]}$ being tight.
\end{remark}
\begin{remark}
Suppose that we have $(L_t)_{t\geq 0}$ and $(Y_t)_{t\geq 0}$ defined on $(\Omega, \mathscr{F}, (\mathscr{F}_t)_{t\geq 0}, \mathbb{P})$ with $\mathscr{F}= \bigvee_{t \geq 0} \mathscr{F}_t$, and that Theorem~\ref{Thm1} is applicable on the truncated space $(\Omega, \mathscr{F}_T, (\mathscr{F}_t)_{t\in [0,T]}, \mathbb{P}\vert_{\mathscr{F}_T})$ for any $T>0$. Then one can sometimes extend it to a locally equivalent measure $\mathbb{Q}$ on $(\Omega, \mathscr{F})$. A probability space having this property is often referred to as being full. An example is the space of all càdlàg functions taking values in a Polish space when equipped with its standard filtration. For more details, see \cite{BichtelerJumps} and \cite{dCriens}.
\end{remark}
Despite of a common structure in (\ref{characteristics}) under (h1) and (h2), the choices of $\alpha$ that we suggest under the different hypotheses in Theorem~\ref{Thm1} differ by their very nature. This is a consequence of different ways of constructing the EMM.

The proof of the existence of an EMM for $(X_t)_{t \in [0,T]}$ consists of two steps. One step is to identify an appropriate possible probability density $Z$, that is, a positive random variable which, given that $\mathbb{E}[Z]=1$, defines an EMM $\mathbb{Q}$ on $(\Omega, \mathscr{F})$ for $(X_t)_{t\in [0,T]}$ through $d\mathbb{Q} = Zd\mathbb{P}$. The candidate will always take the form $Z = \mathscr{E}((\alpha -1)\ast (\mu - \nu))_T$ for some positive predictable function $\alpha$. The remaining step is to check that $\mathbb{E}[Z]=1$ or, equivalently, $\mathscr{E}((\alpha -1)\ast (\mu - \nu))$ is a martingale. Although there exist several sharp results on when local martingales are true martingales, there has been a need for a tractable condition which is suited for the specific setup in question, and this was the motivation for Theorem~\ref{Thm2}. Specifically, it will be used to show Theorem~\ref{Thm1} under hypothesis (h1). As mentioned, the proof of Theorem~\ref{Thm2} is based on a very general approach presented by Lépingle and Mémin \cite{LepingleMemin}.

\begin{theorem}\label{Thm2}
Let $W: \Omega\times [0,T] \times \mathbb{R}\to \mathbb{R}_+$ be a predictable function. Suppose that
\begin{align}\label{Wdominated}
W(t,x) \leq \vert P_t \vert g(x),
\end{align}
where the following hold:
\begin{enumerate}[(a)]
\item The process $(P_t)_{t\in [0,T]}$ is predictable and satisfies that
\begin{enumerate}[(i)]
\item for some fixed $K>0$ and any $t \in [0,T]$, $P_t$ is infinitely divisible with Lévy measure supported in $[-K,K]$, and
\item the collection of random variables $(P_t)_{t \in [0,T]}$ is tight.
\end{enumerate}
\item The function $g: \mathbb{R}\to \mathbb{R}_+$ satisfies $g+g\log(1+g)\in L^1(F)$.
\end{enumerate}
Then $W\ast (\mu - \nu)$ is well-defined and $\mathscr{E}(W \ast (\mu - \nu))$ is a martingale.
\end{theorem}

The following example shows how this result compares to other classical references for measure changes, when specializing to the case where $\mu$ is the jump measure of a Poisson process.
\begin{example}
Suppose that $L$ is a (homogeneous) Poisson process with intensity $\lambda >0$ and consider a density $Z=\mathscr{E}((\alpha - 1)\ast (\mu - \nu))_T$ for some positive predictable process $(\alpha_t)_{t\in [0,T]}$ which paths are integrable on $[0,T]$ almost surely. Within the literature of (marked) point processes, with this setup as a special case, one explicit and standard criterion ensuring that $\mathbb{E}[Z]=1$ is the existence of constants $K_1,K_2>0$ and $\gamma>1$ such that
\begin{align}\label{BremaudCond}
\alpha_t^\gamma \leq K_1 + K_2 (L_t + \lambda t)
\end{align}
for all $t \in [0,T]$ almost surely, see \cite[Theorem~T11~(Ch.~VIII)]{Bremaudpp} or \cite[Eq.~(25)]{gjessing}. We observe that the inequality in (\ref{BremaudCond}) implies that (\ref{Wdominated}) holds with $g\equiv 1$ and $P_t = 2+ K_1 + K_2(L_t + \lambda t)$, $t \in [0,T]$, where $(P_t)_{t\in [0,T]}$ meets (i)-(ii) in Theorem~\ref{Thm2}, thus this criterion is implied by our result. Clearly, this also indicates that we cover other, less restrictive, choices of $(\alpha_t)_{t\in [0,T]}$. For instance, one could take $\gamma =1$ and replace $L$ by any L\'{e}vy process with a compactly supported L\'{e}vy measure in (\ref{BremaudCond}). Note that, although we might have $\alpha(t,x) -1<0$, Theorem~\ref{Thm2} may still be applied according to Remark~\ref{relaxW}. For other improvements of (\ref{BremaudCond}), see also \cite{Sokol}.
\end{example}
Section~\ref{proofs} contains proofs of the statements above accompanied by a minor supporting result and a discussion of the techniques. However, we start by recalling some fundamental concepts which will be (and already has been) used repeatedly.
\section{Preliminaries}\label{prel}
The following consists of a short recap of fundamental concepts. For a more formal and extensive treatment, see \cite{JS}.

The stochastic exponential $\mathscr{E}(M) = (\mathscr{E}(M)_t)_{t \in [0,T]}$ of a semimartingale $(M_t)_{t\in [0,T]}$ is characterized as the unique càdlàg and adapted process with
\begin{align*}
\mathscr{E}(M)_t = 1 + \int_0^t \mathscr{E}(M)_{s-} dM_s
\end{align*}
for $t \in [0,T]$. It is explicitly given as
\begin{align}\label{stochExp2}
\mathscr{E}(M)_t = e^{M_t - M_0 - \tfrac{1}{2}\langle M^c \rangle_t} \prod_{s\leq t} (1+ \Delta M_s)e^{-\Delta M_s}, \quad t \in [0,T],
\end{align}
where $(M^c_t)_{t\in [0,T]}$ is the continuous martingale part of $(M_t)_{t\in [0,T]}$. If $(M_t)_{t \in [0,T]}$ is a local martingale, $\mathscr{E}(M)$ is a local martingale as well. Consequently whenever $\mathscr{E}(M)_t \geq 0$, equivalently $\Delta M_t \geq -1$,  for all $t\in [0,T]$ almost surely, $\mathscr{E}(M)$ is a supermartingale. (Here, and in the following, we have adopted the definition of a semimartingale from \cite{JS}, which in particular means that the process is c\'{a}dl\'{a}g.)

A random measure on $[0,T] \times \mathbb{R}$ is a family of measures $\mu$ such that for each $\omega \in \Omega$, $\mu (\omega ; dt,dx)$ is a measure on $([0,T]\times \mathbb{R}, \mathscr{B}([0,T]) \otimes \mathscr{B}(\mathbb{R}))$ satisfying $\mu (\omega ; \{0\} \times \mathbb{R})=0$. For our purpose, $\mu$ will also satisfy that $\mu (\omega;[0,T] \times \{0\})=0$. Integration of a function $W:\Omega \times [0,T]\times \mathbb{R}\to \mathbb{R}$ with respect to $\mu$ over the set $(0,t]\times \mathbb{R}$ is denoted $W\ast \mu_t$ for $t \in [0,T]$. In this paper, $\mu$ will always be the jump measure of some adapted càdlàg process. To any such $\mu$, one can associate a unique (up to a null set) predictable random measure $\nu$, which is called its compensator. We will always be in the case where $\nu (\omega; dt,dx) = F_t (\omega; dx)dt$ with $(F_t(B))_{t \in [0,T]}$ being a predictable process for every $B \in \mathscr{B}(\mathbb{R})$. One can define the stochastic integral with respect to the compensated random measure $\mu - \nu$ for any predictable function $W: \Omega \times [0,T]\times \mathbb{R} \to \mathbb{R}$ satisfying that $(W^2\ast \mu_t)^{1/2}$, $t \in [0,T]$, is locally integrable. The associated integral process is denoted $W \ast (\mu - \nu)$.

Let $h:\mathbb{R}\to \mathbb{R}$ be a bounded function with $h(x)=x$ in a neighbourhood of $0$. The characteristics of a semimartingale $(M_t)_{t\in [0,T]}$, relative to the truncation function $h$, are then denoted $(C,\nu,B^h)$, which is unique up to a null set. Here $C$ is the quadratic variation of the continuous martingale part of $(M_t)_{t\in [0,T]}$, $\nu$ is the predictable compensator of its jump measure, and $B^h$ is the predictable finite variation part of the special semimartingale given by $M^h_t = M_t - \sum_{s\leq t}[\Delta M_s - h(\Delta M_s)]$ for $t \in [0,T]$. In the case where
\begin{align*}
C_t = \int_0^t c_s ds, \quad \nu(\omega;dt,dx)=F_t(\omega;dx)dt, \quad \text{and} \quad B^h_t = \int_0^tb^h_s ds 
\end{align*}
for suitable predictable processes $(b^h_t)_{t\in [0,T]}$ and $(c_t)_{t\in [0,T]}$ and transition kernel $F_t(\omega;dx)$, we call $(c_t,F_t, b^h_t)$ the differential characteristics of $(M_t)_{t\in [0,T]}$.

Suppose that we have another probability measure $\mathbb{Q}$ on $(\Omega, \mathscr{F})$ such that $d\mathbb{Q} = \mathscr{E}(W\ast (\mu -\nu))_T d\mathbb{P}$, where $\mu$ is the jump measure of an $(\mathscr{F}_t)_{t\in [0,T]}$-L\'{e}vy process $(L_t)_{t \in [0,T]}$ with characteristic triplet $(c, F,b^h)$ relative to a given truncation function $h$ and $\nu$ is the compensator of $\mu$. Then a version of Girsanov's theorem, see \cite{BNSh} or \cite{kallsenDidactic}, implies that under $\mathbb{Q}$, $(L_t)_{t\in [0,T]}$ is a semimartingale with differential characteristics $(c, F_t , b^h_t)$, where
\begin{align}\label{GirsanovF}
F_t (dx) = (1 + W(t,x))F(dx) \quad \text{and}\quad b^h_t = b^h + \int_\mathbb{R} W(t,x) h(x) F(dx).
\end{align}
\section{Proofs}\label{proofs}
In the following, let $f:(-1,\infty)\to \mathbb{R}_+$ be defined by
\begin{align}\label{fFunction}
f(x) = (1+x)\log(1+x)-x, \quad x >-1.
\end{align}
In order to show Theorem~\ref{Thm2} we will state and prove a local version of \cite[Theorem~1~(Section~III)]{LepingleMemin} below.
\begin{lemma}\label{localLM}
Let $(M_t)_{t\in [0,T]}$ be a purely discontinuous local martingale with $\Delta M_t >-1$ for all $t \in [0,T]$ almost surely. Suppose that the process
\begin{align*}
\sum_{s \leq t} f(\Delta M_s), \quad t \in [0,T],
\end{align*}
has compensator $(\tilde{A}_t)_{t\in [0,T]}$ and that there exists stopping times $0=\tau_0 < \tau_1 < \cdots < \tau_n=T$ such that 
\begin{align}\label{stepwiseFinite}
\mathbb{E}\left[\exp \left\{\tilde{A}_{\tau_k} - \tilde{A}_{\tau_{k-1}} \right\}\right] < \infty
\end{align}
for all $k=1,\dots, n$. Then $\mathscr{E}(M)$ is a martingale.
\end{lemma}
\begin{proof}
The following technique of proving the result is similar to the one used in the proof of \cite[Lemma~13]{ProtterShimbo}.

For a given $k \in \{1, \dots, n\}$ define the process
\begin{align*}
M^{(k)}_t = M_{t \wedge \tau_k} - M_{t \wedge \tau_{k-1}}
\end{align*}
for $t \in [0,T]$. Note that $(M^{(k)}_t)_{t\in [0,T]}$ is a (purely discontinuous) local martingale and consequently, $\mathscr{E}\left(M^{(k)}\right)$ is a local martingale. 

Due to the jump structure
\begin{align}\label{jumpStruc}
\Delta M^{(k)}_t =\begin{cases} \Delta M_t & \text{if } t \in (\tau_{k-1},\tau_k] \\
0 & \text{otherwise}
\end{cases}
\end{align} 
it holds that
\begin{align}\label{stoppedJumpProc}
\sum_{s\leq t} f\big(\Delta M^{(k)}_s\big) = \sum_{s\leq t \wedge \tau_k} f(\Delta M_s) - \sum_{s\leq t \wedge \tau_{k-1}} f(\Delta M_s)
\end{align}
for $t \in [0,T]$. Consequently, the compensator of (\ref{stoppedJumpProc}) is $(\tilde{A}_{t \wedge \tau_k} - \tilde{A}_{t \wedge \tau_{k-1}})_{t\in [0,T]}$, and due to the assumption in (\ref{stepwiseFinite}) it follows by \cite[Theorem~8]{ProtterShimbo} that $\mathscr{E}\big(M^{(k)}\big)$ is a martingale. 

By \cite[p.~404]{kallsenshi} we know that for $k \in \{1,\dots, n-1\}$,
\begin{align*}
\mathscr{E}\big(M^{(k)} \big)\mathscr{E}\big(M^{(k+1)} \big) = \mathscr{E}\big(M^{(k)} + M^{(k+1)}+ \big[ M^{(k)}, M^{(k+1)}\big] \big).
\end{align*}
Using (\ref{jumpStruc}) and that $(M^{(k)}_t)_{t\in [0,T]}$ is purely discontinuous, one finds that $\big[M^{(k)}, M^{(k+1)} \big]=0$, so for any $t\in [0, \tau_k]$,
\begin{align*}
\mathscr{E}(M)_t = \mathscr{E}\biggr( \sum_{l=1}^k M^{(l)} \biggr)_t = \prod_{l=1}^k \mathscr{E}\big(M^{(l)}\big)_t.
\end{align*}
Since $\mathscr{E}\big(M^{(l)} \big)_t = \mathscr{E}\big(M^{(l)}\big)_{\tau_{k-1}}$ for all $t \geq \tau_{k-1}$ and $l < k$,
\begin{align*}
\mathbb{E}\left[\mathscr{E}\big(M\big)_{\tau_k}  \right] = \mathbb{E}\biggr[\mathbb{E}\biggr[\mathscr{E}\big(M^{(k)} \big)_{\tau_k} \mid \mathscr{F}_{\tau_{k-1}} \biggr] \prod_{l=1}^{k-1} \mathscr{E}\big(M^{(l)}\big)_{\tau_{k-1}} \biggr] = \mathbb{E}\left[\mathscr{E}\big(M\big)_{\tau_{k-1}}  \right].
\end{align*}
As a consequence, we get inductively that $\mathbb{E}\left[\mathscr{E}(M)_T \right] = \mathbb{E}\left[\mathscr{E}(M)_0\right] =1$. Since $\mathscr{E}(M)$ is a supermartingale, we have the result.
\end{proof}

\begin{proof}[Proof of Theorem~\ref{Thm2}]
We divide the proof into two steps; the first step is to show that assumptions (i)-(ii) on $(P_t)_{t\in [0,T]}$ imply that for $\varepsilon\in (0,1/K)$,
\begin{align}\label{finiteExpect}
\sup_{t \in [0,T]} \mathbb{E}\left[e^{\varepsilon \vert P_t \vert \log (1+ \vert P_t \vert)} \right] < \infty,
\end{align}
and the second step will use this fact to prove that $W\ast (\mu - \nu)_t$, $t\in [0,T]$, is well-defined and that $\mathscr{E}(W\ast (\mu - \nu))$ is a martingale. 

\emph{Step 1:}
The idea is to use a procedure similar to the one in \cite[Lemma~26.5]{Sato} and exploit the tightness property of $(P_t)_{t \in [0,T]}$ to get a uniform result across $t$. In the following we write
\begin{align*}
\Psi_t (u) :=\log \mathbb{E}\big[e^{uP_t}\big]= \frac{1}{2} c_t u^2+ \int_\mathbb{R}(e^{ux} - 1-ux\mathds{1}_{[-1,1]}(x))F_t (dx) + b_t u, \quad u \in \mathbb{R},
\end{align*}
for the Laplace exponent of $P_t$ with associated triplet $(c_t,F_t,b_t)$, $t \in [0,T]$, relative to the truncation function $h(x)= x \mathds{1}_{[-1,1]}(x)$. By the compact support of $F_t$, it follows that $\Psi_t(u)\in \mathbb{R}$ is well-defined for all $u \in \mathbb{R}$ and $t \in [0,T]$ (cf. \cite[Theorem~25.17]{Sato}). For fixed $t$, it holds that $\Psi_t \in C^\infty$,
\begin{align}\label{psiPrime}
\Psi_t'(u) = c_t u+  \int_\mathbb{R} (xe^{ux}-x\mathds{1}_{[-1,1]}(x)) F_t (dx) + b_t,\quad u \in \mathbb{R},
\end{align}
and $\Psi_t'' > 0$, see \cite[Lemma~26.4]{Sato}. From (\ref{psiPrime}) and the inequality $\vert e^{ux} - 1 \vert \leq e^{uK}\vert x \vert$, for $x \in [-K,K]$ and $u \geq 0$, we get the bound
\begin{align}\label{psiPrimeBound}
\Psi_t' (u) \leq c_t u + e^{u K}\int_\mathbb{R} x^2 F_t(dx) + b_t + K F_t((1,K]).
\end{align}
Now suppose that $\sup_{t\in [0,T]} \int_\mathbb{R} x^2 F_t(dx) = \infty$. Then, by the tightness of $(P_t)_{t\in [0,T]}$, we may according to Prokhorov's theorem choose a sequence $(t_n)_{n\geq 1} \subseteq [0,T]$ and a random variable $P$ such that
\begin{align}\label{hypothesis}
P_{t_n} \overset{\mathscr{D}}{\to} P \quad \text{and}\quad \lim_{n \to \infty} \int_\mathbb{R} x^2F_{t_n}(dx) = \infty.
\end{align}
Since $P$ is infinitely divisible, it has an associated characteristic triplet $(c, \rho, b)$. By \cite[Theorem~2.9~(Ch.~VII)]{JS} it holds that
\begin{align*}
\lim_{n \to \infty}\int_\mathbb{R} g d F_{t_n}  =\int_\mathbb{R} g d\rho
\end{align*}
for all $g:\mathbb{R} \to \mathbb{R}$ which are continuous, bounded, and vanishing in a neighbourhood of zero. In particular by the uniformly compact support of $(F_{t_n})_{n\geq 1}$, we get that $\rho$ is compactly supported. As a consequence, \cite[Theorem~2.14~(Ch.~VII)]{JS} and (\ref{hypothesis}) imply that
\begin{align*}
c + \int_\mathbb{R} x^2 \rho (dx) = \lim_{n \to \infty} \left( c_{t_n} + \int_\mathbb{R} x^2 F_{t_n}(dx)\right) = \infty,
\end{align*}
a contradiction, and we conclude that $\sup_{t \in [0,T]} \int_\mathbb{R} x^2 F_t (dx) < \infty$. The same reasoning gives that both $\sup_{t \in [0,T]} c_t$ and $\sup_{t \in [0,T]} (b_t + F_t((1,K]))$ are finite as well. From these observations and (\ref{psiPrimeBound}) we deduce the existence of a constant $C>0$ such that
\begin{align}\label{uniformBound}
\Psi_t'(u) \leq C(1+u+e^{uK})
\end{align}
for $u \geq 0$. 

We may without loss of generality assume that 
\begin{align}\label{infinityAssumption}
\lim_{u\to \pm\infty} \Psi_t'(u) =\infty 
\end{align}
for all $t \in [0,T]$. To see this, let $N^+$ and $N^-$ be standard Poisson random variables which are independent of  each other and of $(P_t)_{t \in [0,T]}$, and consider the process $(\tilde{P}_t)_{t\in [0,T]}$ given by
\begin{align*}
\tilde{P}_t = P_t + K (N^+ - N^-), \quad t \in [0,T].
\end{align*}
This process still satisfies assumptions (i)-(ii) stated in Theorem~\ref{Thm2}, and the derivative of the associated Laplace exponents will necessarily satisfy (\ref{infinityAssumption}) by the structure in (\ref{psiPrime}), since $\tilde{P}_t$ has a L\'{e}vy measure with mass on both $(-\infty,0)$ and $(0,\infty)$. Moreover, the inequality 
\begin{align*}
\mathbb{P}\big(N^+=0\big)^{-2}\sup_{t\in [0,T]} \mathbb{E}\left[e^{\varepsilon \vert \tilde{P}_t\vert \log (1+ \vert \tilde{P}_t \vert)} \right] \geq \sup_{t \in [0,T]} \mathbb{E}\left[e^{\varepsilon \vert P_t \vert \log (1+ \vert P_t \vert)} \right]
\end{align*}
implies that it suffices to show (\ref{finiteExpect}) for $(\tilde{P}_t)_{t\in [0,T]}$. Thus, we will continue under the assumption that (\ref{infinityAssumption}) holds.

Next, by \cite[Lemma~26.4]{Sato} we may find a constant $\xi_0 >0$ such that for any $t$, the inverse of $\Psi_t'$, denoted by $\theta_t$, exists on $(\xi_0,\infty)$ and
\begin{align}\label{inverseBound}
\mathbb{P}(P_t \geq x) \leq \exp\left\{-\int_{\xi_0}^x \theta_t(\xi) d \xi \right\}
\end{align}
for any $x> \xi_0$. Since $\lim_{\xi \to \infty}\theta_t (\xi) = \infty$ and $K-1/\varepsilon' <0$ for $\varepsilon'\in (\varepsilon,1/K)$, it follows by (\ref{uniformBound}) that $\lim_{\xi \to \infty}\xi e^{-\theta_t (\xi)/\varepsilon'} = 0$. In particular by (\ref{uniformBound}) once again, we can choose a $\xi_1\geq\xi_0$ (independent of $t$) such that $-\theta_t(\xi) \leq - \varepsilon'\log \xi$ for every $\xi \geq \xi_1$. Combining this fact with (\ref{inverseBound}) gives that
\begin{align*}
\mathbb{P}(P_t\geq x)\leq \exp\left\{-\varepsilon' \int_{\xi_1}^x \log \xi d \xi \right\} \leq \tilde{C} e^{-\varepsilon'x (\log x - 1)}
\end{align*}
for $x >\xi_1$ and $t \in [0,T]$, where $\tilde{C}$ is some constant independent of $t$. By estimating the probability $\mathbb{P}(P_t \leq -x) = \mathbb{P}(-P_t\geq x)$ in a similar way it follows that $\xi_1$ and $\tilde{C}$ can be chosen large enough to ensure that
\begin{align}\label{GEstimate}
G(x):=\sup_{t \in [0,T]} \mathbb{P}(\vert P_t \vert \geq x) \leq \tilde{C} e^{-\varepsilon' x \log x}
\end{align}
for all $t \in [0,T]$ and $x \geq \xi$. If $G_t(x) = \mathbb{P}(\vert P_t \vert \geq x)$, $x \geq 0$, we get
\begin{align*}
\mathbb{E}\left[e^{\varepsilon \vert P_t \vert \log(1+\vert P_t\vert)}\right] &= -\int_0^\infty e^{\varepsilon x \log (1+x)}dG_t(x) \\
&= 1+ \varepsilon \int_0^\infty e^{\varepsilon x \log (1+x)}\left(\log(1+x) + \frac{x}{1+x} \right)G_t(x)dx 
\end{align*}
using integration by parts, and this implies in turn that
\begin{align*}
\sup_{t \in [0,T]}\mathbb{E}\left[e^{\varepsilon \vert P_t \vert \log(1+ \vert P_t \vert)} \right] \leq 1+ \varepsilon  \int_0^\infty e^{\varepsilon x \log (1+ x)} \left(\log (1+x)+ 1 \right) G(x) dx< \infty
\end{align*}
by (\ref{GEstimate}). Consequently, we have shown that (\ref{finiteExpect}) does indeed hold. 

\emph{Step 2:}
 Arguing that $W\ast (\mu - \nu)_t$, $t \in [0,T]$, is well-defined amounts to showing that $\mathbb{E}[\vert W\vert \ast \nu_T] <\infty$. This is clearly the case since by (\ref{Wdominated}),
\begin{align*}
\mathbb{E}[ W \ast \nu_T] \leq T \sup_{t\in [0,T]} \mathbb{E}[\vert P_t \vert]\int_\mathbb{R} g(x) F(dx),
\end{align*}
and the right-hand side is finite by (\ref{finiteExpect}).

By definition we have the equality
\begin{align*}
f(W)\ast \mu_t = \sum_{s \leq t} f(\Delta (W\ast (\mu - \nu))_s), \quad t \in [0,T],
\end{align*}
and the compensator of the process exists and is given as $\tilde{A} = f(W)\ast \nu$, since 
\begin{align*}
\mathbb{E}[\tilde{A}_T] \leq  &T\biggr[ \Big(1+ \sup_{t\in [0,T]} \mathbb{E}[\vert P_t\vert] \Big)\int_\mathbb{R} g(x)\log(1+ g(x))F(dx) \\
+ &\sup_{t\in [0,T]}\mathbb{E}[\vert P_t \vert \log(1+\vert P_t\vert)]\int_\mathbb{R}g(x)F(dx)\biggr],
\end{align*} 
which is finite by assumption (b) and (\ref{finiteExpect}).

In the following we will argue that (\ref{stepwiseFinite}) in Lemma~\ref{localLM} is satisfied for $\tau_k \equiv t_k$, $k=0,1,\dots,n$, for suitable numbers $0=t_0<t_1<\cdots <t_n=T$, which subsequently allows us to conclude that $\mathscr{E}(W\ast (\mu - \nu))$ is a martingale. Fix $0\leq s < t \leq T$ and note that (\ref{Wdominated}) implies
\begin{align}\label{compensatorBound}
\tilde{A}_t  -\tilde{A}_s \leq \int_s^t \int_\mathbb{R} f(\vert P_u\vert g(x)) F(dx)du,
\end{align}
since $f$ is increasing on $\mathbb{R}_+$. We want to obtain a bound on $h(y) := \int_{\mathbb{R}} f(yg(x))F(dx)$ for $y \geq 0$. 
First note that $f(x) \leq x \log (1+x)$ for any $x \geq 0$ so that we obtain the estimate
\begin{align*}
\frac{h(y)}{y \log y} \leq  \int_\mathbb{R}  g(x) \left[1+  \frac{\log (1+ g(x))}{\log y}\right] F (dx)
\end{align*}
for $y > 1$. Consequently, due to assumption (b), we can apply Lebesgue's theorem on dominated convergence to deduce that 
\begin{align*}
\limsup_{y\to \infty} \frac{h(y)}{y\log(1+y)}< \gamma_1
\end{align*}
for some $\gamma_1 \in (0,\infty)$. By monotonicity of $h$ we may find $\gamma_2 \in (0,\infty)$ so that we obtain the bound $h(y) \leq \gamma_1 y \log (1+y) + \gamma_2$ for all $y \geq 0$. Thus for all $0\leq s < t \leq T$, we have established the estimate
\begin{align}\label{finalEstimate}
\int_s^t h(\vert P_u\vert ) du \leq \gamma_1\int_s^t \vert P_u\vert \log(1+\vert P_u\vert) du + \gamma_2 (t-s).
\end{align}
Now choose a partition $0=t_0<t_1<\cdots < t_n = T$ such that $t_k - t_{k-1}\leq \varepsilon/\gamma_1$, where $\varepsilon$ is a small number such that (\ref{finiteExpect}) holds. By (\ref{compensatorBound}) and (\ref{finalEstimate}) it follows by an application of Jensen's inequality and Tonelli's theorem that
\begin{align*}
e^{-\gamma_2 (t_k - t_{k-1})}\mathbb{E}\Big[e^{ \tilde{A}_{t_k} - \tilde{A}_{t_{k-1}}} \Big] &\leq \mathbb{E}\biggr[\exp \biggr\{\gamma_1 \int_{t_{k-1}}^{t_k} \vert P_t \vert \log(1+ \vert P_t \vert)dt \biggr\} \biggr] \\
&\leq  \frac{1}{t_k - t_{k-1}} \int_{t_{k-1}}^{t_k}\mathbb{E}\left[e^{\varepsilon \vert P_t \vert \log (1+ \vert P_t \vert )} \right] dt \\
&\leq \sup_{t\in [0,T]} \mathbb{E}\left[e^{\varepsilon \vert P_t \vert \log (1+ \vert P_t \vert )}\right],
\end{align*}
which is finite and thus, the proof is completed.
\end{proof}
\begin{remark}\label{relaxW}
It appears from (\ref{compensatorBound}) above that if $F(\mathbb{R})< \infty$ or $W(u,x)=0$ for $x \in (-\delta,\delta)$ with $\delta >0$, one may allow that $W$ takes values in $(-1,\infty)$ by assuming that $\vert W(t,x)\vert \leq \vert P_t\vert g(x)$ and replacing the inequality with 
\begin{align*}
\tilde{A}_t - \tilde{A}_s \leq  M (t-s) + \int_s^t \int_\mathbb{R} f(\vert P_u\vert g(x)) F(dx) du
\end{align*}
for a suitable $M>0$. From this point, one can complete the proof in the same way as above and get that $\mathscr{E}(W\ast (\mu - \nu))$ is a martingale.
\end{remark}
\begin{remark}\label{LMassumptions}
Note that there are other sets of assumptions that can be used to show Theorem~\ref{Thm2}, but they will not be superior to those suggested. Furthermore, the assumptions that we suggest are natural in order to formulate Theorem~\ref{Thm1} in a way which in turn is suited for proving Theorem~\ref{nonGaussianEMM} in the introduction. However, by adjusting the set of assumptions in Theorem~\ref{Thm2}, one may obtain similar adjusted versions of Theorem~\ref{Thm1} (see the discussion in Remark~\ref{thm1assump}). In the bullet points below we shortly point out which properties the assumptions should imply and suggest other choices as well. 
\begin{itemize}
\item The importance of (i)-(ii) is that they ensure (\ref{finiteExpect}) holds. Thus, it follows that one may replace these by $P_t \overset{\mathscr{D}}{=}P_0$ for $t \in [0,T]$ and $\mathbb{E}\big[e^{\varepsilon \vert P_0 \vert \log (1+ \vert P_0 \vert)} \big]< \infty$ for some $\varepsilon >0$.
\item Instead of assuming that $(P_t)_{t\in [0,T]}$ is a process satisfying (\ref{finiteExpect}) and $g + g\log(1+g) \in L^1 (F)$, one may do a similar proof under the assumptions that
\begin{align*}
\sup_{t \in [0,T]} \mathbb{E}\big[e^{\varepsilon \vert P_t\vert^\gamma}\big]<\infty
\end{align*}
and $g^\gamma \in L^1(F)$ for some $\varepsilon >0$ and $\gamma \in (1,2]$. In particular, one may allow for less integrability of $F$ around zero for the cost of more integrability of $(P_t)_{t \in [0,T]}$.
\end{itemize}
\end{remark}
Example~\ref{LMrelax} below shows that one cannot relax assumption (i) in Theorem~\ref{Thm2} and still apply (a localized version of) the approach by Lépingle and Mémin \cite{LepingleMemin}. Moreover, it appears that this approach cannot naturally be improved in the sense of obtaining a weaker condition than (\ref{finiteExpect}), in order to relax assumption (i).
\begin{example}\label{LMrelax}
Consider the case where $W(t,x) = \vert Y x\vert$ for some $\mathscr{F}_0$-measurable infinitely divisible random variable $Y$ with an associated L\'{e}vy measure which has unbounded support. Moreover, suppose that $\mathbb{E}[Y^2]<\infty$ and that $F$ is given such that (b) holds with $g(x) = \vert x \vert$. Then $W\ast (\mu - \nu)$ is well-defined and the compensator of $f(W)\ast \mu$ exists and is given by $f(W)\ast \nu$ (in the notation of (\ref{fFunction})). Following the same arguments as in the proof of Theorem~\ref{Thm2} we obtain that
\begin{align*}
f(W) \ast \nu_t &\geq  c_1 t Y\log(1+ Y) - c_2t
\end{align*}
for $t \in [0,T]$ and suitable $c_1,c_2 >0$. Consequently,
\begin{align*}
\mathbb{E}\Big[e^{f(W)\ast \nu_t} \Big] \geq \mathbb{E}\Big[e^{c_1 t Y \log (1+ Y)} \Big]e^{-c_2 t} = \infty
\end{align*}
for any $t>0$ by \cite[Theorem~26.1]{Sato}. Thus, Lemma~\ref{localLM} cannot be applied if we remove assumption (i) in Theorem~\ref{Thm2}.

Naturally, one can ask if we can find an alternative specification of $f$, that is, does it suffice that $\mathbb{E}[e^{\tilde{f}(W)\ast \nu_t}]< \infty$ for some other measurable function $\tilde{f}:(-1,\infty)\to \mathbb{R}_+$? The idea in the proof of \cite[Theorem~1~(Section~III)]{LepingleMemin} is build on the assumption that $\tilde{f}$ is a function with $(1-\lambda)\tilde{f}(x) \geq 1+ \lambda x - (1+x)^\lambda$ for all $x>-1$ and $\lambda \in (0,1)$. In particular, this requires that
\begin{align*}
\tilde{f}(x) \geq \lim_{\lambda \uparrow 1} \frac{1+ \lambda x - (1+x)^\lambda}{1-\lambda} = f(x)
\end{align*}
for any $x>-1$ and thus, any other candidate function will be (uniformly) worse than $f$.
\end{example}
Before proving Theorem~\ref{Thm1} we will need a small result, which is stated and proven in Lemma~\ref{nuSpec} below. The result may be well-known, however we have not been able to find any references. To a given adapted process $(M_t)_{t\in [0,T]}$ such that for $\omega \in \Omega$, $t \mapsto M_t(\omega)$ is a c\'{a}dl\'{a}g step function, we define its $n$-th jump time and size
\begin{align}\label{jumpNotation}
T_n = \inf \{t \in (T_{n-1},T)\ :\ \Delta M_t \neq 0\}\in (0,T]\quad \text{and}\quad 
Z_n = \Delta M_{T_n},
\end{align}
respectively, for $n\geq 1$. Here we set $T_0\equiv 0$ and $\inf \emptyset = T$.
\begin{lemma}\label{nuSpec} Assume that the jump measure $J$ of some càdlàg adapted process $(M_t)_{t\in [0,T]}$ has a predictable compensator $\rho$ of the form $\rho (dt,dx)= G_t(dx)dt$, where $(G_t(B))_{t \in [0,T]}$ is a predictable process for every $B \in \mathscr{B}(\mathbb{R})$ and $\lambda_t:=G_t(\mathbb{R})\in (0,\infty)$ for $t \in [0,T]$. Then, in the notation of (\ref{jumpNotation}), it holds that
\begin{align}\label{jumpDist}
\mathbb{P} (Z_n \in B \mid \mathscr{F}_{T_n-}) = \Phi_{T_n}(B) \quad \text{on}\quad \{T_n < T\},
\end{align}
for any $n \geq 1$ and $B \in \mathscr{B}(\mathbb{R})$, where $\Phi_t:= G_t/\lambda_t$.
\end{lemma}

\begin{proof}
To show (\ref{jumpDist}), fix $n \geq 1$ and $B \in \mathscr{B}(\mathbb{R})$. Note that $\mathscr{F}_{T_n-}$ is generated by sets of the form $A \cap \{t<T_n\}$ for $t \in [0,T)$ and $A \in \mathscr{F}_t$. Consequently, it suffices to argue that
\begin{align}\label{charRelation}
\mathbb{E}\left[\mathds{1}_{A \cap \{t<T_n < T\}}  \mathds{1}_B(Z_n) \right] =  \mathbb{E}\left[\mathds{1}_{A \cap \{t<T_n<T\}} \Phi_{T_n}(B) \right].
\end{align}
Define the functions $\phi,\psi: \Omega\times [0,T]\times \mathbb{R} \to \mathbb{R}$ by
\begin{align*}
\phi (s,x) = \mathds{1}_{A \cap \{t<T_n \}} \left[\mathds{1}_{\{T_{n-1}\leq t\}} \mathds{1}_{(t,T_n] \times B}(s,x) + \mathds{1}_{\{ T_{n-1}>t\}}\mathds{1}_{(T_{n-1},T_n] \times B}(s,x)\right] \mathds{1}_{(0,T)}(s)
\end{align*}
and
\begin{align*}
\psi (s,x) =\mathds{1}_{A \cap \{t<T_n \}}  \Phi_s (B)\left[\mathds{1}_{\{T_{n-1}\leq t\}} \mathds{1}_{(t,T_n]}(s) + \mathds{1}_{\{T_{n-1}>t\}}\mathds{1}_{(T_{n-1},T_n]}(s)\right]\mathds{1}_{(0,T)}(s),
\end{align*}
and note that they are both predictable. Furthermore, we observe that the functions are defined such that 
\begin{align*}
\phi \ast J_T = \mathds{1}_{A \cap \{t<T_n < T\}}\mathds{1}_B(Z_n), \quad
\psi \ast J_T = \mathds{1}_{A\cap \{t<T_n<T\}} \Phi_{T_n}(B),
\end{align*}
and
\begin{align*}
\phi \ast \rho_T &= \mathds{1}_{A \cap \{t<T_n<T\}} \int_0^T G_s(B)\left[\mathds{1}_{\{T_{n-1}\leq t\}} \mathds{1}_{(t,T_n]}(s) + \mathds{1}_{\{T_{n-1}>t\}} \mathds{1}_{(T_{n-1},T_n]}(s)\right]ds \\
&= \psi\ast \rho_T .
\end{align*}
Using these properties together with the dual relations $\mathbb{E}[\phi \ast J_T] = \mathbb{E}[\phi \ast \rho_T]$ and $\mathbb{E}[\psi \ast \rho_T] = \mathbb{E}[\psi \ast \rho_T]$ we get that (\ref{charRelation}) holds, and this gives the result. 
\end{proof}
Using Lemma~\ref{nuSpec} it follows by a Monotone Class argument that on $\{T_n < T \}$,
\begin{align}\label{condGen}
\mathbb{E} \left[g(Z_n) \mid \mathscr{F}_{T_n-} \right] = \int_{\mathbb{R}} g(x) \Phi_{T_n}(dx)
\end{align}
for any function $g:\Omega \times \mathbb{R} \to \mathbb{R}_+$ which is $\mathscr{F}_{T_n-} \otimes \mathscr{B}(\mathbb{R})$-measurable. With this fact and Theorem~\ref{Thm2} in hand, we are ready to prove Theorem~\ref{Thm1}.
\begin{proof}[Proof of Theorem~\ref{Thm1}]
We prove the result depending on the different hypotheses. In both cases the proof goes by arguing that $\mathscr{E}((\alpha - 1)\ast (\mu - \nu))$ is a martingale and that the probability measure $\mathbb{Q}$ defined by $d\mathbb{Q} = \mathscr{E}((\alpha - 1)\ast (\mu-\nu))_Td\mathbb{P}$ is an EMM for $(X_t)_{t\in [0,T]}$. Since the differential characteristics of $(L_t)_{t\in [0,T]}$ under $\mathbb{P}$ coincide with its characteristic triplet $(c,F,b^h)$, it follows directly from (\ref{GirsanovF}) that if $\mathbb{Q}$ is a probability measure, the differential characteristics of $(X_t)_{t \in [0,T]}$ under $\mathbb{Q}$ are given as in (\ref{characteristics}). In the following we have fixed $a,b>0$ such that (\ref{truncation}) holds.

\emph{Case (h1):}
Consider the specific predictable function $\alpha$ given by (\ref{alphaH1}). Then $\alpha (t,x) \geq 1$ and in particular, $\mathscr{E}((\alpha - 1)\ast (\mu - \nu))_t>0$. Moreover, $\alpha(t,x) - 1 \leq \vert P_t \vert g(x)$ with $P_t =  Y_t + \xi$ and $g(x)=C\mathds{1}_{(a,b)}(\vert x \vert)\vert x \vert$ for some constant $C>0$. Since $\xi$ is just a constant, $(P_t)_{t\in [0,T]}$ inherits the properties in (h1) of $(Y_t)_{t\in [0,T]}$, thus (a) in Theorem~\ref{Thm2} is satisfied. Likewise,
\begin{align*}
\int_\mathbb{R} g[1+ \log(1+ g)]dF = \int_{\{\vert x \vert \in (a,b)\}} C\vert x \vert [1 + \log (1+C\vert x \vert)]F(dx) <\infty,
\end{align*}
showing that (b) is satisfied as well, and we conclude by Theorem~\ref{Thm2} that $\mathscr{E}((\alpha - 1)\ast (\mu - \nu))$ is a martingale. To argue that $(X_t)_{t\in [0,T]}$ is a local martingale under the associated probability measure $\mathbb{Q}$, note that it suffices to show that
\begin{align*}
\int_{\{\vert x \vert \in (a,b)\}} x \alpha(t,x) F(dx) = -(Y_t + b^h)
\end{align*}
by (\ref{characteristics}) and (\ref{measureAffect}) in Remark~\ref{charAnalysis}, where $b^h \in \mathbb{R}$ is the drift component in the characteristic triplet of $L$ with respect to the (pseudo) truncation function $h(x) = x \mathds{1}_{(a,b)^c}(\vert x\vert)$. Thus, we compute
\begin{align*}
\MoveEqLeft \int_{\{\vert x \vert \in (a,b)\}} x \alpha (t,x) F(dx) \\
&= \int_{\{\vert x \vert \in (a,b)\}}xF(dx) +  \frac{(Y_t + \xi)^-}{\sigma_+^2} \int_{(a,b)}x^2F(dx) - \frac{(Y_t+\xi)^+}{\sigma_-^2} \int_{(-b,-a)} x^2F(dx) \\
&=\int_{\{\vert x \vert \in (a,b)\}}  xF(dx) - (Y_t + \xi ) \\
&= \int_{\{\vert x \vert \in (a,b)\}}  xF(dx) - Y_t - \int_{\{\vert x \vert \in (a,b)\}}xF(dx) - b^h \\
&= - (Y_t + b^h),
\end{align*}
and the result is shown under hypothesis (h1).

\emph{Case (h2):}
Set $F^a = F(\cdot \cap [-a,a]^c)$. Note that $F^a((-\infty, \zeta)),F^a((\zeta,\infty))>0$ for $\zeta\in \mathbb{R}$ by assumption, and this implies that we may find a strictly positive density $f_\zeta:\mathbb{R}\to (0,\infty)$ such that
\begin{align}\label{fZeta}
\int_{[-a,a]^c} \frac{f_\zeta (x)}{F^a(\mathbb{R})} F(dx) = 1 \quad \text{and}\quad \int_{[-a,a]^c} x\frac{f_\zeta(x)}{F^a (\mathbb{R})} F(dx) = \zeta.
\end{align}
To see this, assume that $X$ is a random variable on $(\Omega, \mathscr{F}, \mathbb{P})$ with $X \overset{\mathscr{D}}{=} F^a/F^a(\mathbb{R})$. Then, since $\mathbb{E}[X \mid X<\zeta]<\zeta < \mathbb{E}[X \mid X \geq \zeta]$, we may define
\begin{align}\label{lambdaZeta}
\lambda (\zeta)= \frac{\zeta - \mathbb{E}[X \mid X< \zeta]}{\mathbb{E}[X \mid X \geq \zeta] - \mathbb{E}[X\mid X < \zeta]} \in (0,1)
\end{align}
and
\begin{align*}
\varpi_\zeta (B) = (1-\lambda (\zeta)) \mathbb{P}(X \in B \mid X < \zeta) + \lambda (\zeta) \mathbb{P}(X \in B \mid X \geq \zeta)
\end{align*}
for $B \in \mathscr{B}(\mathbb{R})$. We have that $\varpi_\zeta$ is a probability measure which is equivalent to $\mathbb{P}(X \in \cdot)=F^a/F^a(\mathbb{R})$ and has mean $\zeta$. Thus, the density $d\varpi_\zeta /d\mathbb{P}(X \in \cdot)$ is a function that satisfies (\ref{fZeta}). Moreover, such a density is explicitly given by
\begin{align*}
f_\zeta (x) = \frac{1-\lambda (\zeta)}{\mathbb{P}(X<\zeta)} \mathds{1}_{(-\infty,\zeta)}(x) + \frac{\lambda (\zeta)}{\mathbb{P}(X\geq \zeta)}\mathds{1}_{[\zeta,\infty)}(x),\quad x \in \mathbb{R},
\end{align*}
and we see that the map $(x, \zeta)\mapsto f_\zeta (x)$ is $\mathscr{B}(\mathbb{R}^2)$-measurable. By letting
\begin{align*}
\alpha (t,x) = f_{-(Y_t + b^h)/F^a(\mathbb{R})}(x)
\end{align*}
for $\vert x \vert >a$ and $\alpha (t,x) = 1$ for $\vert x \vert \leq a$, we obtain a predictable function $\alpha$, which is strictly positive and satisfies (\ref{alphaH2}).

Thus, it suffices to argue that an $\alpha$ with these properties defines an EMM for $(X_t)_{t\in [0,T]}$ through $\mathscr{E}((\alpha - 1)\ast(\mu - \nu))_T$. First observe that $(\alpha -1)\ast (\mu - \nu)$ is well-defined, since
\begin{align*}
\vert \alpha -1 \vert \ast \nu_T = \int_0^T \int_{[-a,a]^c} \vert \alpha (s,x) - 1\vert F(dx) ds \leq 2F^a(\mathbb{R})T.
\end{align*}
The properties (listed in (h2)) of $\alpha$ imply that $(\alpha - 1)\ast \nu_t=0$ for $t\in [0,T]$. Consequently,
\begin{align*}
\mathscr{E}((\alpha -1)\ast (\mu - \nu))_t=\mathscr{E}((\alpha - 1)\ast \mu )_t = e^{(\alpha - 1)\ast \mu_t + (\log \alpha - (\alpha -1))\ast \mu_t} = \prod_{n=1}^{N_t} \alpha (T_n,Z_n),
\end{align*}
where $(T_n,Z_n)_{n \geq 1}$ is defined as in (\ref{jumpNotation}) for the compound Poisson process $x \mathds{1}_{[-a,a]^c}(x) \ast \mu_t$, $t \in [0,T]$, and $N_t = \mathds{1}_{[-a,a]^c}(x)\ast\mu_t$, $t\in [0,T]$, is the underlying Poisson process that counts the number of jumps. In particular for $n \geq 1$, we have
\begin{align*}
\mathbb{E}\left[\mathscr{E}((\alpha - 1)\ast \mu)_{T_n} \mid \mathscr{F}_{T_{n-1}} \right] &= \mathscr{E}((\alpha - 1)\ast \mu)_{T_{n-1}}\mathbb{E}\left[\mathbb{E}\left[\alpha (T_n,Z_n)  \mid \mathscr{F}_{T_n-}\right] \mid \mathscr{F}_{T_{n-1}} \right]\\
&= \mathscr{E}((\alpha -1)\ast \mu)_{T_{n-1}}
\end{align*}
almost surely by the inclusion $\mathscr{F}_{T_{n-1}}\subseteq \mathscr{F}_{T_n-}$, if we can show that
\begin{align*}
\mathbb{E}\left[\alpha (T_n,Z_n) \mid \mathscr{F}_{T_n-} \right] = 1.
\end{align*}
(Here we recall that $T_0 \equiv 0$.) However, this follows from the observations that $\alpha (T_n,Z_n) = 1$ (since $Z_n=0$) almost surely on the set $\{T_n =T\}$ and on $\{T_n<T\}$,
\begin{align*}
\mathbb{E}[\alpha (T_n,Z_n) \mid \mathscr{F}_{T_n-}]= F^a(\mathbb{R})^{-1} \int_{[-a,a]^c} \alpha (T_n,x)F(dx) = 1
\end{align*}
almost surely by (\ref{alphaH2}) and (\ref{condGen}), since $(\omega, x)\mapsto \alpha(\omega, T_n(\omega),x)$ is $\mathscr{F}_{T_n-}\otimes \mathscr{B}(\mathbb{R})$-measurable. Consequently, $\mathscr{E}((\alpha -1)\ast \mu)_{T_n}$, $n \geq 0$, is a positive $\mathbb{P}$-martingale with respect to the filtration $(\mathscr{F}_{T_n})_{n \geq 0}$ and its mean value function equals $1$, so we may define a probability measure $\mathbb{Q}^n$ on $\mathscr{F}_{T_n}$ by $d\mathbb{Q}^n/d\mathbb{P} = \mathscr{E}((\alpha -1)\ast \mu)_{T_n}$ for $n \geq 1$. By (\ref{GirsanovF}) it follows that the compensator of $\mu$ under $\mathbb{Q}^n$ is $[\alpha(t,x)\mathds{1}_{\{t \leq T_n\}}+ \mathds{1}_{\{t>T_n\}}]F(dx)dt$. From this we get that the counting process $\mathds{1}_{[-a,a]^c}(x) \ast \mu_t$, $t \in [0,T]$, is compensated by
\begin{align*}
\int_0^t\int_\mathbb{R} \mathds{1}_{[-a,a]^c}(x) [\alpha(s,x)\mathds{1}_{\{s \leq T_n\}}+ \mathds{1}_{\{s>T_n\}}]F(dx)ds = F^a(\mathbb{R})t,\quad t \in [0,T],
\end{align*}
under $\mathbb{Q}^n$ using (\ref{alphaH2}). This shows that jumps continue to arrive according to a Poisson process with intensity $F^a(\mathbb{R})$, see e.g. \cite[Theorem 4.5 (Ch. II)]{JS}, which in turn implies that 
\begin{align*}
\mathbb{E}[\mathscr{E}((\alpha -1)\ast \mu)_{T_n}\mathds{1}_{\{T_n<T\}}] = \mathbb{Q}^n(T_n<T)= \mathbb{P}(T_n <T) \to 0
\end{align*}
as $n \to \infty$. As a consequence,
\begin{align*}
1 &= \lim_{n \to \infty} \mathbb{E}[\mathscr{E}((\alpha -1)\ast \mu)_{T_n}\mathds{1}_{\{T_n<T\}}] + \lim_{n \to \infty}\mathbb{E}[\mathscr{E}((\alpha -1)\ast \mu)_T\mathds{1}_{\{T_n= T\}}] \\
&=  \lim_{n \to \infty} \mathbb{E}[\mathscr{E}((\alpha -1)\ast \mu)_{T}\mathds{1}_{\{T_n= T\}}] \\
&= \mathbb{E}[\mathscr{E}((\alpha-1)\ast \mu)_T].
\end{align*}
This shows that $\mathbb{Q}$ defined by $d\mathbb{Q} = \mathscr{E}((\alpha -1)\ast \mu)_Td\mathbb{P}$ is a probability measure on $(\Omega,\mathscr{F})$. To show that $(X_t)_{t \in [0,T]}$ is a local martingale under $\mathbb{Q}$ we just observe that the compensator of $x\mathds{1}_{[-a,a]^c}(x) \ast \mu_t$, $t \in [0,T]$, is given by
\begin{align*}
\int_0^t \int_{[-a,a]^c} x \alpha (s,x)F(dx)ds = - \int_0^t (Y_s + b^h)ds, \quad t \in [0,T],
\end{align*}
according to (\ref{alphaH2}). Thus (\ref{measureAffect}) holds, and the proof is complete by Remark~\ref{charAnalysis}.
\end{proof}
Finally, we use Theorem~\ref{Thm1} to prove Theorem~\ref{nonGaussianEMM}, which was stated in the introduction.

\begin{proof}[Proof of Theorem~\ref{nonGaussianEMM}]
Suppose that $(X_t)_{t\in [0,T]}$ admits an EMM. Then, in particular, it is a semimartingale, and the assumptions imposed imply by \cite[Theorem~4.1~and~Corollary~4.8]{abocjr} that $\varphi$ is absolutely continuous with a density $\varphi'$ satisfying (\ref{densityCondition}). In addition, $(X_t)_{t\in [0,T]}$ has the decomposition
\begin{align}\label{canonicalMA}
X_t = X_0 + \varphi (0) L_t + \int_0^t Y_s ds, \quad t \in [0,T],
\end{align}
where $Y_t = \int_{-\infty}^t \varphi'(t-s)dL_s$. Note that, according to \cite{Cohn} and \cite[Theorem~2.28~(Ch.~I)]{JS}, we may choose $(Y_t)_{t\in [0,T]}$ predictable. From this representation we find that $\varphi (0) \neq 0$, since otherwise an EMM for $(X_t)_{t\in [0,T]}$ would imply $\varphi \equiv 0$. 

Conversely, if $\varphi$ has $\varphi (0) \neq 0$, is absolutely continuous, and the density $\varphi'$ meets (\ref{densityCondition}), we get from \cite[Theorem~4.1~and~Corollary~4.8]{abocjr} that $(X_t)_{t\in [0,T]}$ is a semimartingale that takes the form (\ref{canonicalMA}). Since (\ref{fSupp}) holds, hypothesis (h2) of Theorem~\ref{Thm1} holds for $(\varphi(0) L_t)_{t\in [0,T]}$, and we deduce the existence of an EMM for $(X_t)_{t\in [0,T]}$. Suppose now that (\ref{fSupp}) is not satisfied, but the support of $F$ is bounded and not contained in $(-\infty,0)$ nor in $(0,\infty)$, and the density $\varphi'$ is bounded. Then we observe initially that by \cite{Rosinski_spec}, $(Y_t)_{t\in [0,T]}$ is a stationary process, in particular it is tight, under $\mathbb{P}$, and the law of $Y_0$ is infinitely divisible with Lévy measure given by
\begin{align*}
F^Y(B) = \left(F \times \text{Leb} \right) (\{(x,s) \in \mathbb{R} \times (0, \infty)\ :\ x \varphi' (s) \in B\setminus \{0\}\}),  \quad B \in \mathscr{B}(\mathbb{R}).
\end{align*}
Here Leb denotes the Lebesgue measure on $(0,\infty)$ and $F$ is the L\'{e}vy measure of $(L_t)_{t\in \mathbb{R}}$. In particular, if $C>0$ is a constant that bounds $\varphi'$, we get the inequality
\begin{align*}
F^Y ([-M,M]^c) \leq (F \times \text{Leb})\left(\left[-\tfrac{M}{C},\tfrac{M}{C}\right]^c \times (0,\infty) \right)
\end{align*}
for any $M>0$, and this shows that the Lévy measure of $Y_0$ is compactly supported since the same holds for $F$. In this case (h1) of Theorem~\ref{Thm1} holds, thus we can conclude that an EMM for $(X_t)_{t \in [0,T]}$ exists.
\end{proof}

\begin{remark}
A natural comment is on the existence of $a,b>0$ with the property (\ref{truncation}). In light of the structure of the EMM presented in Theorem~\ref{Thm1}, discussed in Remark~\ref{charAnalysis}, this assumption seems very natural. Indeed, assume that the triplet of $L$ is given relative to the truncation function $h(x) = x\mathds{1}_{[-a,a]}(x)$ and set $\tilde{Y}_t = Y_t + b^h$. Then, according to (\ref{measureAffect}), we try to find $\mathbb{Q}$ under which
\begin{align}\label{pmBalance}
\MoveEqLeft x \mathds{1}_{\{\vert x \vert >a\}}\ast \mu_t + \int_0^t \tilde{Y}_s ds \notag \\  &= \biggr[x\mathds{1}_{\{x > a\}} \ast \mu_t - \int_0^t \tilde{Y}_s^- ds \biggr] - 
\biggr[\vert x \vert\mathds{1}_{\{x<-a\}} \ast \mu_t - \int_0^t \tilde{Y}_s^+ ds \biggr],
\end{align}
$t \in [0,T]$, is a local martingale. Intuitively, $\mathbb{Q}$ should ensure that positive jumps are compensated by the negative drift part and vice versa. Clearly, this construction is not possible if (\ref{truncation}) does not hold for any $a,b>0$ and all jumps are of same sign. If this holds, the construction of $\mathbb{Q}$, if possible, may become rather case specific. For instance, if all jumps of $L$ are positive, one may still make the desired change of measure under (h1) or under the hypothesis that $F$ has unbounded support on $(0,\infty)$, provided that the second term in (\ref{pmBalance}) is not present. Even in the case where the term $\int_0^t (Y_s + b^h)^+ds$, $t \in [0,T]$, is non-zero, it might possibly be absorbed by a change of drift of the Gaussian component in $L$ (if it exists).
\end{remark}
\subsection*{Acknowledgments}
This work was supported by the Danish Council for Independent Research (Grant DFF - 4002 - 00003).
\bibliographystyle{chicago}

\end{document}